\documentclass{amsart}
\usepackage{va, sty-wauto}

\usepackage[pdftex]{hyperref}
\hypersetup{colorlinks,citecolor=blue,linktocpage,hyperindex=true,backref=true}

\usepackage[disable]{todonotes}

\title[Compact moduli of K3 surfaces with automorphism]%
{Compact moduli of K3 surfaces\\ with a nonsymplectic automorphism}
\date{February 15, 2022}

\author{Valery Alexeev}
\email{valery@uga.edu}
\author{Philip Engel}
\email{philip.engel@uga.edu}
\author{Changho Han}
\email{changho.han@uga.edu}
\address{Department of Mathematics, University of Georgia, Athens GA
  30602, USA}

\begin{document}
\begin{abstract}
  We construct a modular compactification via stable slc pairs for the
  moduli spaces of K3 surfaces with a nonsymplectic group of automorphisms under the
  assumption that some combination of the fixed loci of automorphisms
  defines an effective big divisor, and prove that it is semitoroidal.
\end{abstract}
\maketitle
\setcounter{tocdepth}{1}
\tableofcontents

\section{Introduction}

\label{sec:intro}

Let $X$ be a smooth K3 surface over the complex numbers.  An automorphism
$\sigma$ of $X$ is called \emph{non-symplectic} if it has finite
order $n>1$ and $\sigma^*(\omega_X)=\zeta_n \omega_X$, where
$\omega_X\in H^{2,0}(X)$ is a nonzero $2$-form and $\zeta_n$ is a
primitive $n$th root of identity.  By changing the generator of the
cyclic group $\mu_n$ we can and will assume that
$\zeta_n = \exp(2\pi i/n)$. It is well known that a K3 surface
admitting such an automorphism is projective. The possibilities for the
order are the numbers $n$ whose Euler function satisfies
$\varphi(n)\le 20$ with the single exception $n\ne60$, see
\cite[Thm.~3]{machida1998on-k3-surfaces}.

In this paper we study compactification
of moduli spaces of pairs $(X,\sigma)$.
But to begin with, the automorphism group $\Aut(X,\sigma)$, i.e.
those automorphisms of $X$ commuting with $\sigma$,
may be infinite. To fix this, we will usually additionally assume:
\begin{equation}
  \label{eq:g2}
  \tag{$\exists g\ge2$}
  \text{The fixed locus }
  \Fix(\sigma)
  \text{ contains a curve $C_1$ of genus } g\ge2.
\end{equation}

By looking at the $\mu_n$-action on the tangent space of any fixed
point, it is easy to see that $\Fix(\sigma)$ is a disjoint union of
several smooth curves and points. The Hodge index theorem
implies at most one of the fixed curves has genus $g\ge2$.
One could instead have one or two fixed curves of genus
$g=1$. All other fixed curves are isomorphic to $\bP^1$.

Under the \ga{} assumption, the
group $\Aut(X,\sigma)$ is finite. The opposite is almost true.  For
example let $n=2$, i.e.\ $\sigma$ is an involution. Then $\sigma^*$
fixes the Neron-Severi lattice $S_X\subset H^2(X,\bZ)$ and acts as
multiplication by $(-1)$ on the lattice $T_X=S_X^\perp$ of
transcendental cycles. In this case $\Aut(X,\sigma)=\Aut(X)$. 

Deformation classes of such K3
surfaces $(X,\sigma)$ are classified by the primitive $2$-elementary
hyperbolic sublattices $S\subset L_{K3}$. By Nikulin
\cite{nikulin1979integer-symmetric} there are 75 cases, uniquely
determined by certain invariants $(g,k,\delta)$. Among them 51
satisfy \ga. The only case when $|\Aut(X)|<\infty$ but
\ga{} is not satisfied is $(g,k,\delta)=(1,9,1)$ which is
the one-dimensional mirror family to K3 surfaces of degree~$2$. 
In the case $(g,k,\delta)=(2,1,0)$ one has $|\Aut(X)|=\infty$ but
the set $\Fix(\sigma)$ consists of two
elliptic curves, so \ga{} does not hold.

\smallskip

Since the moduli stack of smooth quasipolarized K3 surfaces is
notoriously non-separated, so is usually the moduli stack of smooth
K3s with a nonsymplectic automorphism. For a fixed isometry
$\rho\in O(L_{K3})$ of order $n$, there exists the moduli stack and moduli
space of smooth K3 surfaces ``of type $\rho$'': those pairs $(X,\sigma)$
where the action of $\sigma^*$ on $H^2(X,\bZ)$ can be modeled by $\rho$.
We construct them
in Section~\ref{sec:moduli}.
The maximal separated quotient of $F_\rho$ is
$\dminusrho$, where $\bD_\rho$ is a
symmetric Hermitian domain of type IV if $n=2$ or a complex ball if
$n>2$, $\Gamma_\rho$ is an arithmetic group, and
$\Delta_\rho\subset\bD_\rho$ is the discriminant locus. 

Under the assumption \ga{}, the space
$F_\rho\uade:=\dminusrho$ is the coarse moduli
space for the K3 surfaces $\oX$ with $ADE$ singularities, obtained
from the smooth K3 surfaces $X$ by contracting the $(-2)$-curves
perpendicular to the component $C_1$ with $g\ge2$ in $\Fix(\sigma)$. 
The stack of such $ADE$ K3 surfaces is separated.

\medskip

The main goal of this paper is to construct a functorial,
geometrically meaningful compactification of the moduli space
$F_\rho\uade$, under the assumption \ga. Let $R=C_1$, 
$\varphi_{|mR|}\colon X\to\oX$ be the contraction as above and $\oR$
be the image of $R$.  Then for any $0<\epsilon\ll1$ the pair
$(\oX,\epsilon\oR)$ is a stable pair with semi log canonical
singularities. Then the theory of KSBA moduli spaces (see
\cite{kollar2021families-of-varieties} for the general case or
\cite{alexeev2019stable-pair, alexeev2020compactifications-moduli} for
the much easier special case needed here) gives a moduli
compactification $\oF_\rho\uslc$ to a space of stable pairs with automorphism.

Our main Theorem~\ref{thm:main} says that
$\oF_\rho\uslc$ is a semitoroidal
compactification of $\drho$. This class of compactifications
was introduced by Looijenga \cite{looijenga2003compactifications-defined2}
as a common generalization of Baily-Borel and toroidal compactifications.
As a corollary, the family of $ADE$ K3
surfaces with an automorphism extends along the inclusion
$\dminusrho\hookrightarrow \drho$.

The proof applies a modified form of one of the main theorems of
\cite{alexeev2021compact} about so-called {\it recognizable} divisors.
The $g\ge2$ component of the fixed locus is a canonical choice of a
polarizing divisor. We prove that this divisor is recognizable.

As we point out in Section~\ref{sec:extensions}, the results also
extend to the more general situation of a symmetry group $G\subset\Aut
X$ which is not purely symplectic. 

\smallskip

The cases $n=2,3,4,6$ are of the most interest for compactifications.
If $n\ne 2,3,4,6$ then the space $\drho$ is already compact, see 
\cite{matsumoto2016degeneration} or Corollary~\ref{cor:ords2346}.

K3 surfaces with an involution were classified by Nikulin in
\cite{nikulin1979integer-symmetric}.  K3s with a non-symplectic
automorphism of prime order $p\ge3$ we classified by Artebani, Sarti,
and Taki in \cite{artebani2008non-symplectic,
  artebani2011k3-surfaces}. The case $n=4$ was treated by
Artebani-Sarti in \cite{artebani2015symmetries} and the case $n=6$ by
Dillies in \cite{dillies2009order6, dillies2012on-some}.

We note two cases where our KSBA, semitoroidal compactification
$\oF_\rho\uslc$ is computed in complete detail:
Alexeev-Engel-Thompson \cite{alexeev2019stable-pair} for the case of K3
surfaces of degree 2, generically double covers of $\bP^2$, and
a forthcoming work
Deopurkar-Han~\cite{deopurkar2021stable-quadrics} which treats a
$9$-dimensional component in the moduli for $n=3$.

\smallskip

The paper is organized as follows. In Section~\ref{sec:moduli} we set
up the general theory of the moduli of K3 surfaces with a
non-symplectic automorphisms.
In Section~\ref{sec:stable-pair-compactifications} we define the
stable pair compactifications and prove the main
Theorem~\ref{thm:main}.
In Section~\ref{sec:quotient-surfaces} we relate K3 surfaces with
nonsymplectic automorphisms with their quotients $Y=\oX/\mu_n$, and the
compactification $\oF_\rho\uslc$ with the KSBA compactification of the
moduli spaces of log del Pezzo pairs~$(Y,\frac{n-1+\epsilon}{n} B)$.

In Section~\ref{sec:extensions} we extend the results in two different
ways: to K3 surfaces with a finite group of symmetries $G\subset\Aut
X$ which is not purely symplectic, and to more general polarizing
divisors associated with such a group action.

Throughout, we work over the field of complex numbers.

\begin{acknowledgements}
  The first author was partially supported by NSF under DMS-1902157.
\end{acknowledgements}

\section{Moduli of K3s with a nonsymplectic automorphism}
\label{sec:moduli}

\subsection{Notations}
\label{sec:notations}

A lattice is a free abelian group with an integral-valued symmetric
bilinear form.  Let $\lk=H^{\oplus 3}\oplus E_8^{\oplus 2}$ 
be a fixed copy of the even unimodular lattice of signature
$(3,19)$, where $H=\II_{1,1}$ corresponds to the bilinear form
$b(x,y)=xy$ and $E_8$ is the standard negative definite even lattice
of rank $8$. For any smooth K3 surface $X$ the cohomology lattice
$H^2(X,\bZ)$ is isometric to $\lk$.

Denote by $S=S_X$ the Neron-Severi lattice $\Pic(X)=\NS(X)$.
By the Lefschetz $(1,1)$-theorem, it equals
$(H^{2,0}(X))^\perp \cap H^2(X,\bZ) \subset H^2(X,\bC)$. We have
$H^{2,0}(X)=\bC\omega_X$ for some nowhere vanishing holomorphic two-form
$\omega_X$.
If $X$ is projective, then
$S_X$ is nondegenerate of signature $(1,r_X-1)$. In this case,
its orthogonal complement
$T_X=(S_X)^\perp\subset H^2(X,\bZ)$ is the {\it transcendental lattice},
of signature $(2,20-r_X)$. 
The \emph{K\"ahler cone} $\cK_X\subset H^{1,1}(X,\bR)$ is the set of
classes of K\"ahler forms on $X$; it is an open convex cone.

\begin{theorem}[Torelli Theorem for K3 surfaces, \cite{piateski-shapiro1971torelli}]
  \label{thm:torelli} The isomorphisms $\sigma\colon X'\to X$ are in bijection with the isometries
$\sigma^*\colon H^2(X,\bZ)\to H^2(X',\bZ)$ satisfying 
the conditions $\sigma^*(H^{2,0}(X))=H^{2,0}(X')$
and $\sigma^*(\cK_X)=\cK_{X'}$. \end{theorem}

For any lattice $H$, a \emph{root} is a vector $\delta\in H$ with
$\delta^2=-2$. The set of all roots is denoted by $H\mt$.  The Weyl
group $W(H)$ is the group generated by reflections
$v \mapsto v+(v,\delta)\delta$ for $\delta\in H\mt$.  It is a normal
subgroup of the isometry group $O(H)$.

\subsection{Moduli of marked unpolarized K3s}
\label{subsec:moduli-unpolarized}

The basic reference here is \cite{asterisque1985geometrie-des-surfaces}. 
Let $X$ be a K3 surface. A \emph{marking} is an isometry
$\phi\colon H^2(X,\bZ)\to \lk$.
Let
\begin{displaymath}
  \bD = \bP\{ x \in \lkc \mid x\cdot x=0,\ x\cdot\bar x>0\},
  \quad \dim \bD = 20.
\end{displaymath}
There exists a fine moduli space $\cM$ of \emph{marked K3 surfaces}
and a period map $\pi\colon\cM\to\bD$,
$(X,\phi)\mapsto \phi(H^{2,0}(X)) \in \bP(\lkc)$. $\cM$ is
a non-Hausdorff $20$-dimensional complex
manifold with two isomorphic connected
components interchanged by negating $\phi$. The period map
is \'etale and surjective.

For a period point $x\in\bD$, the vector space
$(\bC x\oplus \bC \bar x)\cap \lkr\subset \lkc$ is
positive definite of rank $2$ and its orthogonal
complement $x^\perp \cap\lkr$ has signature $(1,19)$. Let
\begin{displaymath}
  \{ v\in x^\perp\cap \lkr \mid v^2>0 \} = P_x \sqcup (-P_x)
\end{displaymath}
be the two connected components of the set of positive square vectors.
 Then the fiber $\pi\inv(x)$ is identified with the set
of connected components $\cC$ of 
\begin{equation}\label{eq:cC-chambers}
  \big( P_x\sqcup (-P_x) \big) \setminus
  \cup_\delta\,\delta^\perp
  \ \text{for}\ \delta\in (x^\perp\cap \lk)\mt.
\end{equation}

Namely, an open chamber $\cC$ is identified with the K\"ahler cone
$\cK_X$ of the corresponding marked K3 surface $X$ via the marking
$\phi$. The connected components
are permuted by the reflections and $\pm\id$, and $\pi\inv(x)$ is a
torsor under the group $\bZ_2\times W_x$, where $W_x=W(x^\perp\cap \lk)$. Since
$x^\perp\cap \lkr$ is hyperbolic, the group and the fiber
$\pi\inv(x)$ may be infinite. For a general point $x\in\bD$, the
lattice $x^\perp\cap \lk$ has no roots and the fiber $\pi\inv(x)$ consists of two
points, one in each connected component of~$\cM$.

\subsection{Moduli of $\rho$-marked and $\rho$-markable K3 surfaces
  with automorphisms}
\label{sec:lattice polarized}

Fix $\rho\in O(\lk)$ an isometry of order $n>1$ and consider
a K3 surface $X$ with a non-symplectic automorphism $\sigma$
of order $n$.

\begin{definition}\label{def:rho-marking}
  A {\it $\rho$-marking} of
  $(X,\sigma)$ is an isometry $\phi:H^2(X,\bZ)\to \lk$ such that
  $\sigma^* = \phi\inv\circ\rho\circ\phi$.  We say that $(X,\sigma)$ is
  {\it $\rho$-markable} if it admits a $\rho$-marking.

  A family of $\rho$-marked surfaces is a smooth morphism
  $f\colon (\cX,\sigma_B)\to B$ with an automorphism
  $\sigma_B\colon\cX\to\cX$ over $B$, together with an isomorphism of
  local systems $\phi_S\colon R^2f_*\underline{\bZ} \to
  L\otimes\underline{\bZ}_B$ such that 
  every fiber is a K3 surface with a $\rho$-marking.  A family
  $f\colon (\cX,\sigma_B)\to B$ is $\rho$-markable if such an
  isomorphism exists locally in complex-analytic topology on $B$.

  We define the moduli stacks $\cM_\rho$ of $\rho$-marked,
  resp. $F_\rho$ of $\rho$-markable K3 by taking $\cM_\rho(B)$,
  resp. $F_\rho(B)$ to be the groupoids of such families
  over base $B$.
\end{definition}

\begin{definition}
  Define $L_\bC^{\zeta_n}$ to be the eigenspace $x\in L_\bC$ such that
  $\rho(x)=\zeta_nx$, and the subdomain
  $\bD_\rho=\bP(L_\bC^{\zeta_n})\cap \bD\subset \bD$. Define
  $\Gamma_\rho\subset O(\lk)$ as the group of changes-of-marking:
  $\Gamma_\rho := \{\gamma \in O(L)\mid \gamma\circ \rho = \rho\circ
  \gamma\}$.
\end{definition}

\begin{definition} Let the {\it generic transcendental lattice}
  $T_\rho:=\lkc^{\rm prim}\cap \lk$ be the intersection of $\lk$ with the
  sum of all primitive eigenspaces of $\rho$, and let the {\it generic
    Picard lattice} be $S_\rho=(T_\rho)^\perp$.
  Let $L^G=\Fix(\rho)\subset S_\rho$ be classes in $\lk$ fixed by
  $\rho$.  (Here, we use $G=\langle\rho\rangle\simeq\bZ_n$ to avoid
  confusing notation, as $L^G$ would be.)
\end{definition}

  Note that the $\zeta_n$-eigenspaces $\lkc^{\zeta_n}$ and
  $T_{\rho,\,\bC}^{\zeta_n}$ coincide, and that for any
  K3 surface with a $\rho$-marking the two fixed sublattices
  $\phi\colon S_X^{G} = H^2(X,\bZ)^{G} \isoto \lk^G$
  are identified.

For there to exist a $\rho$-markable algebraic K3 surface,
the signature of $T_\rho$ must be $(2,\ell)$ for some $\ell$,
as there is necessarily a vector of positive norm fixed by $\sigma^*$
(the sum of a $\sigma^*$-orbit of an ample class). The converse is
also true.

When $n=2$, we have that $\bD_\rho\subset \bP(T_{\rho,\bC})$ is (two
copies of) 
the Type IV domain associated
to the lattice $T_\rho$. When $n\geq 3$, the condition that $x \cdot x=0$
is vacuous on $\bD_\rho$ because $x\cdot y=0$ for eigenvectors
$x,y$ of $\rho$ with non-conjugate eigenvalue. Thus,
\begin{displaymath}
  \bD_\rho =\bP\{x\in T_{\rho,\,\bC}^{\zeta_n}  \mid x\cdot \bar x>0\}  
\end{displaymath}
is a complex ball, a Type I domain. The Hermitian form $x\cdot \bar y$ on $T_{\rho,\,\bC}^{\zeta_n}$
necessarily has signature $(1,\ell)$ for some $\ell$ for there to exist a $\rho$-markable
K3 surface.

\begin{definition}
The discriminant locus is 
$\Delta_\rho:=(\cup_\delta \,\delta^\perp)\cap  \bD_\rho $ ranging
over all roots $\delta$ in $(L^G)^\perp$.
\end{definition}

\begin{lemma}\label{lem:im-pi-rho}
  Let $\rho\in O(L)$ be an isometry of order $n>1$. Then
  \begin{enumerate}
  \item A marking $\phi\colon H^2(X,\bZ)\to L$      defines a
    $\rho$-marking, i.e. defines an automorphism $\sigma$ such that
    $\sigma^*=\phi\inv\circ\rho\circ\phi$ iff the period $x=\pi((X,\phi))$
    lies in $\bD_\rho\setminus\Delta_\rho$ and there exists an ample
    line bundle $\cL_h$ on $X$ with $h=\phi(\cL_h)\in L^G$.
  \item For a point $x\in\bD_\rho\setminus\Delta_\rho$ the set of
    $\rho$-marked K3s with this period is a torsor over the group
    $\Gamma_\rho\cap (\bZ_2\times W_x)$.
  \end{enumerate}
\end{lemma}
\begin{proof}
  Because the action is nonsymplectic, $\rho(x)=\zeta_nx\ne x$. 
  For any $h\in L^G$ one has $\rho(h)=h$, which implies that
  $hx=0$. Thus, $L^G\perp x$ and $S_X^G\simeq L^G$. 
  
  Clearly, one must have $x\in\bD_\rho$. By the Torelli theorem,
  automorphism $a=\phi\inv \circ \rho\circ\phi$ of $H^2(X,\bZ)$ is
  induced by an automorphism $\sigma$ of $X$ iff it sends the K\"ahler
  cone $\cK_X$ to itself. By averaging, this is equivalent to having
  an $a$-invariant K\"ahler class $\cL_h\in\cK_X\cap H^2(X,\bZ)$.
  And since
  $L^G\perp x$, one has $\cL_h\perp \omega_X$, so $\cL_h\in
  S_X$ and $\cL_h$ is an ample line bundle. This proves (1).

  If $x\perp\delta$ for some root $\delta\in(L^G)^\perp$ then
  $\cL_\delta=\phi\inv(\delta)\in\Pic(X)$ and either $\cL_\delta$ or
  $\cL\inv_\delta$ is effective.
  Then for $\cL_h$ as in
  part~(1) one has both $\cL_h\cdot\cL_\delta=0$ because
  $h\perp\delta$ and $\cL_h\cdot\cL_\delta\ne0$ because $\cL_h$ is
  ample.  Contradiction.

  On the other hand, let $x\in\bD_\rho\setminus\Delta_\rho$. Then
  $\lk^G\not\subset \cup_\delta\, \delta^\perp$ for
  ${\delta\in(x^\perp\cap \lk)\mt}$.  Thus, there exists a chamber
  $\cC$ in $P_x\setminus\cup_\delta\,\delta^\perp$ such that
  $\cC\cap \lk^G\ne\emptyset$. Let $(X,\phi)$ be the K3 surface
  corresponding to this chamber. Then there exists
  $h\in \cC\cap \lk^G$ and by part~(1) the marking $\phi$ is a
  $\rho$-marking. 

  Any surface with the same period $x$ is isomorphic to $X$, but with
  a marking $\phi' = g\circ \phi$ for some $g\in\bZ_2\times W_x$. Then
  one has both $\sigma^*=\phi\inv\circ\rho\circ\phi$ and
  $\sigma^*=(\phi')\inv\circ\rho\circ\phi'$ iff $g\in
  \Gamma_\rho$. This proves~(2).
\end{proof}

\begin{lemma}
  There exists a fine moduli space $\cM_\rho$ of $\rho$-marked K3
  surfaces with a non-symplectic automorphism. $\cM_\rho$ an open subset
  of $\pi\inv(\bD_\rho\setminus\Delta_\rho).$
\end{lemma}
\begin{proof}
  The points of $\cM$ are chambers $\cC$ in
  Equation~\eqref{eq:cC-chambers} over
  $x\in \bD_\rho\setminus\Delta_\rho$. As in the proof of
  Lemma~\ref{lem:im-pi-rho}, one has $\cC\in\cM_\rho$ iff
  $\cC\cap L^G\ne\emptyset$. This is an open condition.
\end{proof}

The restriction of $\pi\colon\cM\to\bD$ gives the period map
$\pi_\rho\colon\cM_\rho\to \bD_\rho\setminus\Delta_\rho$.
The general fiber of $\pi_\rho$ is a torsor over $\Gamma_\rho\cap
(\bZ_2\times W(S_\rho))$. Thus, $\cM_\rho$ is not separated iff there
exists $x\in\bD_\rho\setminus\Delta_\rho$ such that $\Gamma_\rho\cap
W_x \supsetneq \Gamma_\rho\cap W(S_\rho)$. This indeed happens:

\begin{example}\label{ex:triple-cover}
  Consider the $9$-dimensional family of $\mu_3$-covers
  of $\bP^1\times\bP^1$ branched in a curve $B$ of bidegree $(3,3)$,
  studied by Kond\={o} \cite{kondo2002moduli-of-curves-of-genus-4}. In
  this case,
  \begin{displaymath}
    S_\rho=L^G = \big(\Pic (\bP^1\times\bP^1)\big)(3) = H(3)
    \quad\text{and}\quad
    T_\rho=(L^G)^\perp=H \oplus H(3)\oplus E_8^2.
  \end{displaymath}
  Let $\oY$ be a degeneration of the quadric
  $\bP^1\times\bP^1\subset\bP^3$ to a quadratic cone and $\oX\to\oY$
  be the $\mu_3$-cover branched in a curve $\oB\in|\cO_\oY(3)|$ not
  passing through the apex.  Let $Y=\bF_2$ and $X$ be the minimal
  resolutions of $\oY$ and $\oX$. The $\bP^1$-fibration on $Y$ gives
  an elliptic fibration on $X$, and the preimage of the $(-2)$-section
  of $Y$ is a union of three disjoint $(-2)$-sections $e$,
  $\sigma e$, $\sigma^2e$ on $X$, interchanged by the automorphism
  $\sigma$. The invariant sublattice
  $S_X^\sigma = \big(\Pic(\bF_2)\big)(3)=H(3)$ is generated by $f$ and
  $f'=f+\sum_{i=0}^2\sigma^ie$.

  The only $(-2)$-curves on $X$ are $\sigma^ie$ and they do not lie in
  $(S_X^\sigma)^\perp$. Thus, once we fix a marking $\phi$, the period
  $x$ of $X$ will be in $\bD_\rho\setminus\Delta_\rho$. The reflections
  $w_i$ in the roots $\rho^i\phi(e)$ commute. Their product
  $w=w_0w_1w_2$ is non-trivial: on $L^G$ it acts as the reflection 
  that interchanges $\phi(f)$ and $\phi(f')$.
  It is easy to check that $w\in\Gamma_\rho$.  So
  $\Gamma_\rho\cap W_x\ne 1$ and $W(L^G)=1$.
  
  Thus, the map $\cM_\rho\to\bD_\rho\setminus\Delta_\rho$ is not
  separated in this case. Locally it looks like the ``double-headed
  snake'' $\bA^1\cup_{\bA^1\setminus 0}\bA^1 \to \bA^1$ times~$\bA^8$.
  Here is another way to see the same. The positive cone $P$ in
  $H(3)_\bR$ is the unique Weyl chamber for the Weyl group
  $W\big(H(3)\big)=1$; its rays are $\phi(f)$ and $\phi(f')$. The
  hyperplane $\phi(e)^\perp$ cuts it in half.  The intersections of
  the Weyl chambers $\cC\subset P_x\setminus\cup \delta^\perp$ of
  Equation~\ref{eq:cC-chambers} with $P$ are either halves of $P$.
\end{example}





\begin{theorem}
  \label{separated-rho}
  The moduli stack $F_\rho$ of $\rho$-markable K3 surfaces with a
  non-symplectic automorphism is the quotient
  $F_\rho=\mathcal{M}_\rho/\Gamma_\rho$. Its coarse moduli space
  admits a bijective period map to $\dminusrho$, and the coarse moduli
  space of the separated quotient $F_\rho^{\rm sep}$ is $\dminusrho$.
  The generic stabilizer is the group
  \begin{displaymath}
    K_\rho:=\ker(\Gamma_\rho\to \Aut(\bD_\rho))
    / \Gamma_\rho\cap(\bZ_2\times W(S_\rho))    
  \end{displaymath}

\end{theorem}

\begin{proof} The statement is immediate from the definitions and the
  above two Lemmas by quotienting the period map $\pi_\rho$.  The
  points of $\pi_\rho\inv(x)$ are permuted by $\Gamma_\rho$, thus they
  are identified in the $\Gamma_\rho$-quotient. They are also
  identified in the separated quotient.

  For $\rho$ to correspond to any K3 surface with a nonsymplectic
  automorphism, $S_\rho$ must have signature $(1,r-1)$ for some $r$,
  and for $T_\rho$ to have signature $(2,20-r)$. The action of
  $\Gamma_\rho$ on the Type IV domain $\bD(T_\rho)$ factors through
  $O(T_\rho)$ and is therefore properly discontinuous. Thus, the
  action of $\Gamma_\rho$ on $\bD_\rho$ is properly discontinuous,
  and so $\dminusrho$ is makes sense as a complex-analytic space.  (It
  is also quasiprojective by Baily-Borel.)

  The last statement follows
  from Lemma~\ref{lem:im-pi-rho}(2) by noting that for a generic
  $x\in\bD_\rho\setminus\Delta_\rho$ one has $x^\perp\cap L=S_\rho$.
\end{proof}

\begin{remark}\label{rem:dolgachev-kondo}
  The proof of part (1) of Lemma~\ref{lem:im-pi-rho}
  and of Theorem~\ref{separated-rho}
  follow the arguments of Dolgachev-Kondo
  \cite[Thms.~11.2, 11.3]{dolgachev2007moduli-of-k3}. Sections 10 and 11 of
  \cite{dolgachev2007moduli-of-k3} contain a construction of the
  moduli space of K3 surfaces with a non-symplectic automorphism that
  is based on moduli of lattice polarized K3s. But it uses
  \cite[Thm.~3.1]{dolgachev1996mirror-symmetry} which unfortunately is
  false, as was noted in \cite{alexeev2021compact} and as
  Example~\ref{ex:triple-cover} also shows.  For this reason, we
  decided to give an alternative construction.
\end{remark}

\begin{remark}\label{W-quotient}
  Even though the map to $\dminusrho$ in Theorem~\ref{separated-rho}
  is bijective, the coarse moduli space of $F_\rho$ is a non-separated
  algebraic space when $\cM_\rho$ is not separated. This is very
  similar to the algebraic space obtained by dividing a ``two-headed
  snake'' $\bA^1\cup_{\bA^1\setminus 0}\bA^1$ by the involution $z\to-z$ 
  exchanging the heads. The quotient is a non-separated algebraic
  space with a bijection to $\bA^1=\bA^1/\pm$.

  We note that the separated quotient
  $F_\rho^{\rm sep}$ is a stack
  $[\bD_\rho\setminus \Delta_\rho:_W\Gamma_\rho]$ which can be locally
  constructed near $x\in \bD_\rho\setminus \Delta_\rho$ by first
  taking a coarse quotient by the normal subgroup
  $\Gamma_\rho\cap (\bZ_2\times W_x)\unlhd {\rm Stab}_x(\Gamma_\rho)$
  and then taking the stack quotient by
  ${\rm Stab}_x(\Gamma_\rho)/\Gamma_\rho\cap (\bZ_2\times W_x)$.  See
  \cite[Rem.~2.36]{alexeev2021compact}.
\end{remark}

 \begin{proposition}\label{finite-aut}
 Suppose $\sigma\in {\rm Aut}(X)$ fixes a curve $R$ of genus at least $2$,
 i.e. the assumption \ga{} holds.
 Then ${\rm Aut}(X,\sigma)$ is finite. \end{proposition}
 
 \begin{proof} Let $h\in {\rm Aut}(X,\sigma)$ be an automorphism of $X$ satisfying
 $h\circ\sigma = \sigma\circ h$. Then $h$ permutes the fixed components of $\sigma$.
 Since there is at most one component $R$ of genus $g\ge2$, we conclude $h(R)=R$.
 Hence $h\in {\rm Aut}(X,\mathcal{O}(R))$, a finite group.
 \end{proof}
 
 Note that generic stabilizer $K_\rho$ from
 Theorem~\ref{separated-rho} is never the trivial group, as
 $\rho\in K_\rho$ is a nontrivial element. As this is the automorphism
 group of a generic element $(X,\sigma)\in F_\rho$, if \ga{} holds
 then $K_\rho$ is finite by Proposition~\ref{finite-aut}.

 \begin{example} Consider the double cover 
$\pi\colon X\to \mathbb{P}^2$ branched
 over a smooth sextic $B$. There is a non-symplectic involution $\sigma$ switching
 the two sheets of $X$, acting on $H^2(X,\bZ)$ by fixing $h=c_1(\pi^*\mathcal{O}(1))$
 and negating $h^\perp$. Choosing a model $\rho$ for the action of $\sigma^*$
 on cohomology, we have that $S_\rho=\langle 2\rangle$ and
 $T_\rho=\langle -2\rangle\oplus H^{\oplus 2}\oplus E_8^{\oplus 2}$ are the $(+1)$-
 and $(-1)$-eigenspaces, respectively.

The divisor $\Delta_\rho/\Gamma_\rho\subset \mathbb{D}_\rho/\Gamma_\rho = F_2$
has two irreducible components corresponding to $\Gamma_\rho$-orbits of roots
$\delta\in (T_\rho)_{-2}$. Such an orbit is uniquely determined by the divisibility ($1$ or $2$)
of $\delta\in T_\rho^*$. The case where the divisibility is $2$ corresponds 
to when $B$ acquires a node.
Then there is an involution $\sigma$ on
the minimal resolution of the double cover $X\to \oX\to \mathbb{P}^2$,
but $\sigma^*(\delta)=\delta$, $\sigma^*(h)=h$ and the $(+1,-1)$-eigenspaces of $\sigma^*$
have dimensions $(2,20)$.
Thus, no $\rho$-marking can be extended over a family $\mathcal{X}\to C$ 
with central fiber $X$ and general fiber as above.

When the divisibility of $\delta$ is $1$, $\mathbb{P}^2$ degenerates to $\mathbb{F}_4^0=\bP(1,1,4)$
and the minimal resolution of the double cover $X\to \oX\to \mathbb{F}_4^0$ is an elliptic
K3 surface with $\sigma$ the elliptic involution. Again the eigenspaces have dimension profile
$(2,20)$ and so $(X,\sigma)$ is not $\rho$-markable for the $\rho$ as above.
 \end{example}

\section{Stable pair compactifications}
\label{sec:stable-pair-compactifications}

\subsection{Complete moduli of stable slc pairs}
\label{sec:slc}

We refer the reader to
\cite[Sec.~2B]{alexeev2020compactifications-moduli} and
\cite[Sec.~7D]{alexeev2021compact} for a detailed discussion of
stable K3 surface pairs and their compactified moduli. Briefly:

\begin{definition}
  In our context, a stable slc surface pair is a pair $(S,\epsilon D)$,
  where
\begin{enumerate}
\item $S$ is a connected, reduced, projective Gorenstein surface $S$
  with $\omega_S\simeq\cO_S$ which has semi log canonical
  singularities.
\item $D$ is an effective ample Cartier divisor on $S$ that does not
  contain any log canonical centers of $S$.
\end{enumerate}
\end{definition}

Then for sufficiently small rational number $\epsilon>0$ the pair
$(S,\epsilon D)$ is stable, meaning:
\begin{enumerate}
\item it has semi log canonical singularities, and 
\item the $\bQ$-Cartier divisor $K_S +\epsilon D$ is ample. 
\end{enumerate}
``Sufficiently small'' works in families: for a fixed $D^2$ there
exists $\epsilon_0$ so that if a pair $(S,\epsilon D)$ is stable in
the above definition for some $\epsilon$ then it is stable for any
$0<\epsilon\le \epsilon_0$. 

The main application to K3 surfaces is an observation that for any K3
surface $\oX$ with ADE singularities and an effective ample divisor
$\oR$, the pair $(\oX,\epsilon \oR)$ is stable. Indeed,
$\omega_\oX\simeq\cO_\oX$, the surface $\oX$ has canonical
singularities---which is much better than semi log canonical---and
there are no log centers.

As usual, let $F_{2d}$ denote the moduli space of polarized K3
surfaces $(\oX,\oL)$ with ADE singularities and ample primitive line
bundle $\oL$ of degree $\oL^2=2d$, and $P_{2d,m}\to F_{2d}$ denote
the moduli space of pairs $(\oX,\epsilon\oR)$ with an effective divisor
$\oR\in |m\oL|$. Then the main result for K3 surfaces is the following:

\begin{theorem}\label{thm:slc-moduli}
\begin{enumerate}
\item For the stable pairs as above there exists an algebraic
  Deligne-Mumford moduli stack $\cM\uslc$, with a coarse moduli
  space $M\uslc$.
\item The closure $\oP_{2d,m}\uslc$ of $P_{2d,m}$ in $M\uslc$ is
  projective and provides a  compactification of $P_{2d,m}$
  to a moduli space of stable slc pairs.
\end{enumerate}
\end{theorem}

To apply this result to a compactification of $F_\rho^{\rm sep}$ one
needs to choose, in a canonical manner, a big and
nef divisor on the generic $(X,\sigma)\in F_\rho$.

\begin{definition}
  A {\it canonical choice of polarizing divisor} is an algebraically
  varying big and nef divisor $R$ defined over a Zariski dense subset
  $U\subset F_\rho$ of the moduli space of $\rho$-markable K3 surfaces. 
\end{definition}

\subsection{Stable pair compactification of $F_\rho^{\rm sep}$}
\label{sec:slc-F}

We apply Theorem~\ref{thm:slc-moduli} to construct a stable pair
compactification in the present context as follows.

Suppose that for each surface $(X,\sigma)\in F_\rho$ 
assumption~\ga{} holds, i.e. the fixed locus $\Fix(\sigma)$
contains a component $C_1$ of genus $g\ge2$, as well as possibly
several smooth rational curves $C_i$ and some isolated points. In fact,
it suffices that a single $(X,\sigma)\in F_\rho$ satisfies assumption~\ga{}
because the genus of $C_1$ is constant in a family of smooth K3 surfaces
with non-symplectic automorphism. So $R=C_1$ gives a
canonical choice of polarizing divisor for all of $U=F_\rho$.

Let $\pi\colon X\to\oX$ be the contraction to an ADE K3 surface such
that the divisor $\oR:=\pi(C_1)$ is ample; it has degree
$\oR^2=2g(C_1)-2>0$. It provides us with an ample
divisor on $\oX$.
If $\cO(\oR)=\oL^m$ for a primitive $\oL$ then the pair
$(\oX,\cO(\oR))$ is a point of $F_{2d,m}$ and the pair
$(\oX,\epsilon\oR)$ is a point of $P_{2d,m}$.

\begin{definition}\label{def:slc_M-pol_nonsymp}
  We define the map $\psi\colon  F_\rho \to P_{2d,m}$ as
  follows. Pointwise, it sends $(X,\sigma)$ to $(\oX,\epsilon\oR)$. 
  In every flat family $f\colon \cX\to S$
  of K3 surfaces with automorphism, the sheaf
  $\cO_\cX(\cR)$ is relatively big and nef. Since $R^i\cL^d=0$
  for $i>0$, $d>0$, it gives a contraction to a flat family
  $\bar f\colon(\ocX,\ocR)\to S$.
  This induces the map on moduli.
\end{definition}

\begin{lemma}\label{lem:finite_slc}
  The map $\psi\colon  F_\rho \to P_{2d,m}$ defined above induces an
  injective map 
  $F_\rho^{\rm sep}\to {\rm im}(\psi)$.
\end{lemma}
\begin{proof}
The map $\psi$ factors through the separated
  quotient of $F_\rho$ because $P_{2d,m}$ is separated.
  Now suppose there is an isomorphism of pairs $\overline{f}\colon
  (\oX_1,\oR_1)\to (\oX_2,\oR_2)$ inducing an isomorphism of the
  minimal resolutions $f\colon (X_1,R_1)\to (X_2,R_2)$. Consider the
  morphism $\varphi = \sigma_1\inv  f\inv  \sigma_2 
  f$. Then $\varphi$ is a \emph{symplectic} automorphism of $X_1$
  fixing the curve $R_1$ pointwise.
  Since $\varphi$ preserves $\cO_{X_1}(R_1)$, it has finite order.
  By \cite{nikulin1979finite-groups} the fixed set of a
  finite order symplectic K3 automorphism is finite. 
  Thus, $\varphi=\id$ and $f$ preserves the group action.
  So, $(X,\sigma)$ is uniquely determined by $(\oX,\oR)$.
\end{proof}

\begin{remark}\label{sep-moduli}
  $F_\rho^{\rm sep}$ itself has a moduli interpretation: It is the
  moduli space $F_\rho\uade$ of $ADE$ K3 surfaces $(\oX,\overline{\sigma})$ with
  automorphism, for which ${\rm Fix}(\overline{\sigma})$ is ample, and for
  which the minimal resolution $(X,\sigma)\to (\oX,\overline{\sigma})$ is
  $\rho$-markable. 
\end{remark}

\begin{definition}
  Let $Z= {\rm im}(\psi)$ and let $\oZ$ be its closure in
  $\oP_{2d,m}\uslc$, with reduced scheme structure. The stable pair compactification
  $$F_\rho^{\rm sep} = F_\rho\uade \hookrightarrow \oF_\rho\uslc$$ is defined
  as the normalization of $\oZ$.
\end{definition}


In particular, $\oF_\rho\uslc$ is normal by definition.  Points
correspond to the pairs $(\oX,\epsilon\oR)$, possibly degenerate, with
some finite data.

\subsection{Kulikov degenerations of K3 surfaces}
\label{sec:kulikov}

A basic tool in the study of degenerations of K3 surfaces is Kulikov
models. We use them in the argument below, so we briefly
recall the definition.

Let $(C,0)$ denote the germ of a smooth curve at a
point $0\in C$ and let $C^*=C\setminus 0$. Let $X^*\to C^*$ be
a family of algebraic K3 surfaces. 

\begin{definition}
A {\it Kulikov model} $X\to (C,0)$ is an extension of $X^*\to C^*$ for
which $X$ is a smooth algebraic space,
$K_X\sim_C 0$, and $X_0$ has reduced normal crossings.
We say the $X$ is {\it Type I, II, or III}, respectively, depending
on whether $X_0$ is smooth, has double curves but no triple
points, or has triple points, respectively. We call the central
fiber $X_0$ of such a family a {\it Kulikov surface}.
\end{definition}

A key result on the degenerations of K3 surfaces
is the theorem of Kulikov \cite{kulikov1977degenerations-of-k3-surfaces}
and Persson-Pinkham \cite{persson1981degeneration-of-surfaces}:

\begin{theorem}
Let $Y^*\to C^*$ be a family of algebraic K3 surfaces.
Then there is a finite base change $(C',0)\to (C,0)$ and a sequence of
birational modifications of the pull back $Y'\dashrightarrow X$
such that $X$ has smooth total space, $K_X\sim_{C'} 0$, and $X_0$ has
reduced normal crossings.
\end{theorem}

We recall some fundamental results about Kulikov models.
The primary reference is \cite{friedman1986type-III}.
Let $T:H^2(X_t,\bZ)\to H^2(X_t,\bZ)$ denote the Picard-Lefschetz
transformation associated to an oriented simple loop in
$C^*$ enclosing $0$. Since $X_0$ is reduced normal crossings,
$T$ is unipotent. Let $$N:=\log T = (T-I)-\tfrac12 (T-I)^2 +\cdots$$
be the logarithm of the monodromy. 

\begin{theorem}\label{thm:pic-lef}\cite{friedman1986type-III}\cite{friedman1984a-new-proof}
Let $X\to (C,0)$ be a Kulikov model. We have that
 \begin{enumerate} 
\item[] if $X$ is Type I, then $N=0$,
\item[] if $X$ is Type II, then $N^2=0$ but $N\neq 0$,
\item[] if $X$ is Type III, then $N^3=0$ but $N^2\neq 0$.
\end{enumerate} 
The logarithm of monodromy is integral, and of the form
$Nx = (x\cdot \lambda)\delta-(x\cdot \delta)\lambda$ for $\delta\in H^2(X_t,\bZ)$
a primitive isotropic vector, and $\lambda\in \delta^\perp/\delta$ satisfying
$$\lambda^2=\#\{\textrm{triple points of }X_0\}.$$ When $\lambda^2=0$,
its imprimitivity is the number of double curves of $X_0$.
\end{theorem}

Thus, the Types I, II, III of Kulikov model are distinguished by the
behavior of the {\it monodromy invariant} $\lambda$: either
$\lambda=0$, $\lambda^2=0$ but $\lambda\neq 0$, or $\lambda^2\neq 0$ respectively.

\begin{definition}\label{J-latt}
Let $J\subset H^2(X_t,\bZ)$ denote the primitive
isotropic lattice $\bZ\delta$ in Type III or the saturation of
$\bZ\delta\oplus \bZ\lambda$ in Type II. 
\end{definition}

\subsection{Baily-Borel compactification}
\label{sec:baily-borel}

Let $N$ be a lattice of signature $(2,\ell)$, together with an isometry
$\rho\in O(N)$ of finite order $n$, such that all eigenvalues of $\rho$
on $N_\bC$ are primitive $n$th roots of unity, and $N_\bC^{\zeta_n}$
contains a vector $x$ of positive Hermitian norm $x\cdot \bar x$.
This is the situation which arises for a non-symplectic automorphism
of an algebraic K3 surface, with $N=T_\rho$. Then we have a Type IV
domain
$$\bD_N= \bP\{x\in N_{\bC}\mid x\cdot x=0,\ x\cdot \bar x>0\}$$
For $n=2$ one has $\bD_\rho=\bD_N$. For $n>2$ one has a Type I
subdomain of $\bD_N$
$$\bD_\rho= \bP\{x\in N_{\bC}^{\zeta_n}\mid x\cdot \bar x>0\}$$
$\bD_\rho$ admits the action of the arithmetic group
$\widetilde{\Gamma}_\rho:=\{\gamma\in O(N)\mid \gamma\circ \rho = \rho\circ \gamma\}$.
Fix a finite index subgroup $\Gamma \subset \widetilde{\Gamma}_\rho$.

Recall that $\bD_N$ and $\bD_\rho$ embed into their compact duals
$\bD_N^c$, $\bD_\rho^c$, which are defined by dropping the condition
that $x\cdot \bar x>0$. Define $\obD_N\subset \bD_N^c$,
$\obD_\rho\subset \bD_\rho^c$ as their topological closures.
One has a well known description of the rational boundary components
of $\bD_N$, see e.g. see
\cite{looijenga2003compactifications-defined2}. 

\begin{definition}
  A {\it rational boundary component} of $\bD_N$ is an analytic subset
  $B_J\subset~\obD_N$ of the form: 
\begin{enumerate}
\item $(\bP J_\bC \setminus \bP J_\bR)\cap \obD_N$ for ${\rm rk}\,J=2$
a primitive isotropic sublattice of $N$,
\item $\bP J_\bC \cap \obD_N$ for ${\rm rk}\,J=1$ a primitive
isotropic sublattice of $N$.
\end{enumerate}

The rational boundary components of $\bD_\rho$ are intersections of
$B'_J=B_J\cap\obD_\rho$.

One defines the {\it rational closure} of $\bD_\rho$ to be
$\obD_\rho\ubb:=\bD_\rho\cup_J B'_J,$,
topologized via a horoball topology at the boundary.
Then the {\it Baily-Borel compactification} of $\bD_\rho/\Gamma$ is (at least topologically)
$\overline{\bD_\rho/\Gamma}\ubb := \obD_\rho\ubb/\Gamma$.
\end{definition}

The space $\overline{\bD_\rho/\Gamma}\ubb$ was shown to have the structure
of a projective variety by Baily-Borel
\cite{baily1966compactification-of-arithmetic}.
For Type IV
domains $\bD_N$ and $\bD_\rho$ if $n=2$, the boundary components (1)
are isomorphic to 
$\bH\sqcup (-\bH)$ and the boundary components (2) are points.
For $n>2$, the boundary components of the Type I domain $\bD_\rho$ are
points. If $\rk J=2$ then a point $[x]\in B_J$ corresponds to the elliptic
curve $E_x=J_\bC/(J+\bC x)$.

\begin{lemma}\label{lem:ords346} In the case $n>2$, for each boundary
  component $B_J'$ we necessarily have ${\rm rk}\,J=2$ and
  $n\in \{3,4,6\}$, and $x\in B'_J$ corresponds to the elliptic curve
  with $j(E_x)=0$ if $n=3$ or $6$, and with $j(E_x)=1728$ if $n=4$.
\end{lemma}
\begin{proof}
  If $B'_J$ is boundary component of $\bD_\rho$ then
  $N_\bC^{\zeta_n}\cap J_\bC\ne0$. Since $J$ is defined over $\bZ$ and
  $\zeta_n\notin\bR$, then $N_\bC^{\overline{\zeta}_n}\cap J_\bC\ne0$
  as well.  This implies that $\rk J=2$ and
  $$J_\bC = J_\bC^{\zeta_n}\oplus J_\bC^{\overline{\zeta}_n}.$$
  Thus, $\rho(J_\bC)=J_\bC$, implying that $\rho(J)=J$.
  Therefore $\rho\big{|}_J\in\GL(J)\cong\GL_2(\bZ)$ necessarily has
  order $n$. Thus, $n\in \{3,4,6\}$.  For a point $[x]\in B_J'$ one has
  $x\in N_\bC^{\zeta_n}$ and $\mu_n\subset\Aut(E_x)$. This
  determines~$E_x$.
\end{proof}

\begin{corollary}\label{cor:ords2346}
  If $n\ne 2,3,4,6$ then the rational closure of $\bD_\rho$ is
  simply $\bD_\rho$ itself. So $\bD_\rho/\Gamma$ is already compact.
\end{corollary}

The following is a
well-known consequence of Schmid's nilpotent orbit theorem:

\begin{proposition} Let $X^*\to C^*$ be a degeneration
of a $\rho$-markable K3 surfaces over a punctured analytic disk $C^*$.
A lift of the period mapping $\widetilde{C^*}\cong \bH\to \bD_\rho$
approaches the Baily-Borel cusp $B_J$ as ${\rm Im}(\tau)\to \infty$,
where $J$ is the monodromy lattice in $H^2(X_t,\bZ)$, cf. 
Definition \ref{J-latt}. When ${\rm rk}(J)=2$, the limiting point $x\in B_J$ corresponds to
an elliptic curve $E_x$ isomorphic to any double curve of the central
fiber $X_0$ of a Kulikov model $X\to C$. \end{proposition}
 
 \begin{corollary}\label{cor:ords2346-again} If $n\neq 2,3,4,6$, any degeneration
 of $(X,\sigma)\in F_\rho$ has Type I. If $n\in \{3,4,6\}$,
 any degeneration of $(X,\sigma)\in F_\rho$
 has Type I or II.
 \end{corollary}

 The last statement was also proved by Matsumoto
 \cite{matsumoto2016degeneration} using different techniques. His
 proof also holds in some prime characteristics.

\subsection{Semitoroidal compactifications}
\label{sec:semitoroidal}

Semitoroidal compactifications of arithmetic quotients $\bD/\Gamma$
for type IV Hermitian symmetric domains $\bD$ were defined by Looijenga
\cite{looijenga2003compactifications-defined2} (where they were called
``semitoric''). They simultaneously generalize toroidal and
Baily-Borel compactifications of $\bD/\Gamma$.
The case of the complex ball $\bD$ (a type I symmetric Hermitian
domain) is comparatively trivial. The semitoroidal compactifications
in this case are implicit in
\cite{looijenga2003compactifications-defined1,
  looijenga2003compactifications-defined2}. We quickly overview
  the construction in both cases now.

\begin{definition}\label{semitor}
A {\it $\Gamma$-admissible semifan} $\mathfrak{F}$ consists of the following data:

When $n=2$, it is a convex, rational, locally polyhedral decomposition $\mathfrak{F}_J$ 
of the rational closure $\mathcal{C}^+(J^\perp/J)$ of the positive norm vectors, 
for all rank $1$ primitive isotropic sublattices
$J\subset N$, such that:
\begin{enumerate}
\item $\{\mathfrak{F}_J\}_{J\subset N}$ is $\Gamma$-invariant. In particular,
a fixed $\mathfrak{F}_J$ is invariant under the natural action of ${\rm Stab}_J(\Gamma)$ on $\mathcal{C}^+(J^\perp/J)$.
\item A compatibility condition of the $\{\mathfrak{F}_J\}_{J\subset N}$
along any primitive isotropic lattice $J'\subset N$ of rank $2$ holds,
see Definition \ref{compatible}.
\end{enumerate}

When $n>2$, the data is much simpler: It consists, for each primitive
isotropic sublattice $J\subset N$ satisfying $J_\bC\cap N_\bC^{\zeta_n}\neq \emptyset$,
of a primitive sublattice $\mathfrak{F}_J\subset J^\perp/J$ such that the collection
$\{\mathfrak{F}_J\}$ is $\Gamma$-invariant.  
\end{definition}

\begin{definition}\label{compatible} Let $J'\subset N$ be primitive isotropic of rank $2$.
We say that the collection $\{\mathfrak{F}_J\}_{J\subset N}$ is {\it compatible
along $J'$} if, given any primitive sublattice $J\subset J'$ of rank $1$,
the kernel of the hyperplanes of
$\mathfrak{F}_J$ containing $J'/J$, when intersected with $(J')^\perp/J\subset J^\perp/J$
and then descended to
$(J')^\perp/J'$, cut out a fixed sublattice $\mathfrak{F}_{J'}\subset (J')^\perp/J'$ 
which is independent of $J$.\end{definition}

In both the $n=2$ and $n>2$ cases, we use the same notation
$\mathfrak{F}:=\{\mathfrak{F}_J\}_{J\subset N}$ even though 
$J$ ranges over rank $1$ isotropic sublattices when $n=2$
and ranges over rank $2$ isotropic sublattices when $n>2$.

In the Type IV case, Looijenga constructs a compactification
$\bD/\Gamma\hookrightarrow \overline{\bD/\Gamma}^{\mathfrak{F}}$
for any $\Gamma$-admissible semifan $\mathfrak{F}$,
so consider the Type I case.
By Lemma~\ref{lem:ords346} we may restrict to $n\in \{3,4,6\}$. There is a $\bZ[\zeta_n]$-lattice
$$Q:=(N\otimes_\bZ \bZ[\zeta_n])^{\zeta_n}\subset N_\bC^{\zeta_n}=Q_\bC$$
on which Hermitian form $x\cdot \overline{y}$ defines a $\bZ[\zeta_n]$-valued
Hermitian pairing of signature $(1,\ell)$ for some $\ell$.
Any element of $\widetilde{\Gamma}_\rho$ (in particular,
any element of $\Gamma$) preserves $Q$ and the Hermitian form on it. The converse
also holds. Thus $\Gamma\subset U(Q)$ is a finite index subgroup
of the group of unitary isometries of $Q$ and $\Gamma_\bR = U(Q_\bC)=U(1,\ell)$.
The boundary components
$B_J=\bP(J_\bC^{\zeta_n})$ are then projectivizations of the
isotropic $\bZ[\zeta_n]$-lines $K\subset Q$. Here $K_\bC=J_\bC^{\zeta_n}$.

Choose a generator $k\in K$.
Then any $[x]\in \bD_\rho\subset \bP Q_\bC$ has a unique representative
$x\in Q_\bC$ for which $k\cdot x=1$. This realizes $\bD_\rho$ as a
generalized tube domain in the affine hyperplane $V_k:=\{k\cdot x=1\}\subset Q_\bC$.

Let $U_K\subset {\rm Stab}_K(\Gamma)$ be the unipotent
subgroup (i.e. $U_K$ acts on $K$, $K^\perp/K$, and $Q/K^\perp$ by the identity).
Then $U_K$ acts on $V_k$ by translations.
Choosing some isotropic $k'\in Q_\bC$ for which $k'\cdot k=1$, any
element $x\in V_k$ can be written uniquely as $x=k'+x_0+ck$ for some $x_0\in \{k,k'\}^\perp$
and $c\in \bC$. The image of $\bD_\rho$ is exactly those $x$
satisfying $2{\rm Re}(c)> -x_0\cdot \bar x_0$.

The fibration $\bD_\rho\to K_\bC^\perp/K_\bC$ 
sending $x\mapsto x_0\textrm{ mod }K_\bC$ is a fibration of right half-planes.
The action of $U_K$ fibers over the action of a translation
subgroup $\overline{U}_K\subset K^\perp/K$ on $K_\bC^\perp/K_\bC$ and thus, 
there is a fibration $$\bD_\rho/U_K\to (K_\bC^\perp/K_\bC)/\overline{U}_K=:A_K$$
over an abelian variety. The fibers are quotients of the right half-planes
with coordinate $c$ by a discrete, purely imaginary, translation group
isomorphic to $\bZ$. This realizes $\bD_\rho/U_K$ is a punctured holomorphic disc bundle over 
$A_K$.

\begin{definition} $\bD_\rho/U_K$ is the {\it first partial quotient} associated to the
Baily-Borel cusp $K$. The extension of this punctured disc bundle to a disc bundle
$\overline{\bD_\rho/U_K}^{\rm can}\to A_K$ for a given $K$
is called the {\it toroidal extension at the cusp $K$}. \end{definition}

We will identify the divisor at infinity, i.e. the zero section of the disc bundle,
with $A_K$ itself.

\begin{construction} The {\it toroidal compactification} of $\bD_\rho/\Gamma$ is constructed as follows:
Let $\Gamma_K$ be the finite group defined by the exact sequence
$$0\to U_K\to {\rm Stab}_K(\Gamma)\to \Gamma_K\to 0.$$
For each cusp $K$, quotient the toroidal extension
$$V_K:=\overline{\bD_\rho/U_K}^{\rm can}/\Gamma_K\supset \bD_\rho/{\rm Stab}_K(\Gamma).$$ 
A well-known theorem states that there exists a horoball neighborhood $\bP K_\bC\in N_K\subset \bD_\rho\ubb$
such that $(N_K\setminus \bP K_\bC)/{\rm Stab}_K(\Gamma)\hookrightarrow \bD_\rho/\Gamma$ injects.
Thus, we can glue a neighborhood of the boundary $A_K/\Gamma_K\subset V_K$ to $\bD_\rho/\Gamma$, ranging
over all $\Gamma$-orbits of cusps $K$. The result is the toroidal compactification
$\overline{\bD_\rho/\Gamma}^{\rm tor}$.
 \end{construction}

The boundary divisors of $\overline{\bD_\rho/\Gamma}^{\rm tor}$ are in bijection
with $\Gamma$-orbits of isotropic $\bZ[\zeta_n]$-lines $K\subset Q$ and 
the boundary divisor is isomorphic to $A_K/\Gamma_K$, where $\Gamma_K$
acts by a subgroup of the finite group $U(K^\perp/K)$. There is a morphism
$$\overline{\bD_\rho/\Gamma}^{\rm tor}\to \overline{\bD_\rho/\Gamma}\ubb$$
which contracts each boundary divisor to a point. As such, the normal bundle
of the boundary divisor is anti-ample. Passing to a finite index subgroup
$\Gamma_0\subset \Gamma$, we can assume that $\Gamma_K$ is trivial
for all cusps $K$ and the anti-ampleness still holds.
This proves that the normal bundle to $A_K\subset \overline{\bD_\rho/U_K}^{\rm can}$
in the first partial quotient is anti-ample.

Using \cite{grauer1962uber_modifikationen} one shows that
a divisor in a smooth analytic space, isomorphic to an abelian variety and
with anti-ample normal bundle, can be contracted along any abelian subvariety.
In particular, for any sub-$\bZ[\zeta_n]$-lattice $\mathfrak{F}_K\subset K^\perp/K$,
there is a contraction $$\overline{\bD_\rho/U_K}^{\rm can}\to \overline{\bD_\rho/U_K}^{\rm \mathfrak{F}_K}$$
which is an isomorphism away from the boundary divisor and contracts exactly
the translates of the abelian subvariety ${\rm im}(\mathfrak{F}_K)_\bC\subset A_K$.

To construct the semitoroidal compactification $\overline{\bD_\rho/\Gamma}^\mathfrak{F}$,
we wish to glue, at each cusp $K$, a punctured analytic open neighborhood of the boundary
of $\overline{\bD_\rho/U_K}^{\mathfrak{F}_K}/\Gamma_K$ to $\bD_\rho/\Gamma$. This
is only possible if the action of $\Gamma_K$ on $\overline{\bD_\rho/U_K}^{\rm can}$
descends along the above contraction.
The condition in Definition \ref{semitor} ensures that the collection
$\mathfrak{F}=\{\mathfrak{F}_K\}$
is $\Gamma$-invariant. So an individual
$\mathfrak{F}_K$ is $\Gamma_K$-invariant and the $\Gamma_K$ action
descends. Thus, we have constructed
the semitoroidal compactification.

\begin{remark} A feature of the construction is that one can pull back
a semifan $\mathfrak{F}$ for a Type IV domain to any Type I subdomain,
and there will be a morphism between the corresponding semitoric compactifications.
\end{remark}

\subsection{Recognizable divisors}
\label{sec:recognizable}

We recall the main new concept ``recognizability" introduced in
\cite{alexeev2021compact}. We slightly modify the definition as necessary
for moduli spaces of K3 surfaces with $\rho$-markable
automorphism:

\begin{definition}\label{recog-def}
A canonical choice of polarizing divisor $R$ for
$U\subset F_\rho$ is {\it recognizable} if for every
Kulikov surface $X_0$ of Type I, II, or III which smooths to
some $\rho$-markable K3 surface,
there is a divisor $R_0\subset X_0$ such that on {\it any} smoothing
into  $\rho$-markable K3 surfaces $X\to (C,0)$ with $C^*\subset U$,
the divisor $R_0$ is, up to the action of ${\rm Aut}^0(X_0)$,
the flat limit of $R_t$ for $t\neq 0\in C^*.$
\end{definition}

We use the term ``smoothing" to mean specifically a Kulikov model $X\to (C,0)$.
Roughly, Definition \ref{recog-def} amounts to saying that the
canonical choice $R$ can also be made on any Kulikov surface,
including smooth K3s. 

\begin{theorem}\label{main-thm-recog}
  If $R$ is recognizable, then $\oF_\rho\uslc$ is semitoroidal
  compactification of $F_\rho$ for a unique semifan $\mathfrak{F}_R$.
\end{theorem}

\begin{proof} The proof when $n=2$ is essentially the same
as \cite[Thm.~1.2]{alexeev2021compact}. So we restrict
our attention to the Type I case $n>2$, which is ultimately much
simpler anyways. First, we show that $\oF_\rho\uslc$ contains
$\bD_\rho/\Gamma_\rho$.

 Let $\cM_\rho^*$ be the closure
of the moduli space of $\rho$-marked K3 surfaces $\cM_\rho$
in the space of all marked K3 surfaces $\cM$ and let
$F_\rho^*=\cM_\rho^*/\Gamma_\rho$ be the quotient.
Given any smooth K3 surface $X_0\in F_\rho^*\setminus U$,
the recognizability implies that the universal family
$(\mathcal{X}^*,\mathcal{R}^*)\to U$ extends over $F_\rho^*$ 
by the same argument as  \cite[Prop.~6.3]{alexeev2021compact}.
Thus, the argument of Lemma \ref{lem:finite_slc} shows
that there is a morphism
$(F_\rho^*)^{\rm sep}=\bD_\rho/\Gamma_\rho\to P_{2d,m}$ and so  
we may as well have constructed $\oF_\rho\uslc$ by taking the normalization
of the closure of the image of $\bD_\rho/\Gamma_\rho$, which is notably
already normal.
This completes the proof when $n\neq 3,4,6$.

So let $\bP K_\bC$ be a Baily-Borel cusp of $\bD_\rho$ when $n\in \{3,4,6\}$.
We observe that the closure of $\bD_\rho /U_K$ in the toroidal extension 
$\bD(J)\subset \bD(J)^\lambda$ of the ``universal" first partial quotient
for unpolarized K3 surfaces, cf. \cite[Def.~4.18]{alexeev2021compact},
is simply the first partial quotient $\overline{\bD_\rho/U_K}^{\rm can}$. 
 \cite[Prop.~4.16]{alexeev2021compact} shows that $\bD(J)$
 embeds into a family of affine lines over $J^\perp/J\otimes_\bZ \widetilde{\mathcal{E}}$
 where $\widetilde{\mathcal{E}}$ is the universal elliptic curve over $\bH\sqcup (-\bH)$
 and $\bD(J)^\lambda$ is its closure in a projective line bundle.
 The space $\bD_\rho/U_K$ sits inside this affine line bundle as the inverse image of
 $$K^{\perp \textrm{ in }Q}/K\otimes_{\bZ[\zeta_n]}
 E\subset J^\perp/J\otimes_\bZ \widetilde{\mathcal{E}}$$
 where $E$ is the elliptic curve admitting an action of $\zeta_n$
 (note that $K=J$ but with the additional structure of a $\bZ[\zeta_n]$-lattice).
 
Thus we may restrict a Type II {\it $\lambda$-family}, cf.
\cite[Def.~5.34]{alexeev2021compact},
to a family $$\mathcal{X}\to \overline{\bD_\rho/U_K}^{\rm can}$$
of Kulikov surfaces of Types I + II. We call $\mathcal{X}$ a {\it $K$-family}.
Note that any $K$-family admits a birational automorphism
which is the action of the automorphism $\sigma$ on the restriction of $\mathcal{X}$ to
$(\bD_\rho\setminus \Delta_\rho)/U_K$.

The arguments in \cite[Secs.~6,\,8]{alexeev2021compact},
leading up to the proof of Theorem 1.2 of {\it loc. cit.} now all apply
to $K$-families $\mathcal{X}$, showing that there is a sandwich
of normal compactifications $$\overline{\bD_\rho/\Gamma_\rho}^{\rm tor}\to \oF_\rho\uslc \to
\overline{\bD_\rho/\Gamma_\rho}\ubb.$$ Using that the normal image
of an abelian variety is an abelian variety (a similar argument is used
in \cite[Thm.~7.18]{alexeev2021compact}), we conclude that there must
exist a $\Gamma_\rho$-admissible semifan $\mathfrak{F}_R$ for which 
$\oF_\rho\uslc=\overline{\bD_\rho/\Gamma_\rho}^{\rm \mathfrak{F}_R}$.
\end{proof}

\subsection{The main theorem}
\label{sec:main-thm}

\begin{theorem}
  \label{thm:main}
  Under the assumption~\ga{}, $R=C_1$ is recognizable
  for $F_\rho$. The stable pair
  compactification $\oF_\rho\uslc$  is a semitoroidal compactification
  of $\bD_\rho/\Gamma_\rho$.
\end{theorem}

\begin{proof}
By Theorem \ref{main-thm-recog}, the second statement follows from the first.
 Let $(X,R)\to(C,0)$ be a Kulikov model with a flat family of divisors $R\subset X$
  for which \begin{enumerate}
  \item there is an automorphism $\sigma$ on $X^*\to C^*$ making
  $(X_t,\sigma_t)\in F_\rho$ for $t\neq 0$,
\item  $R_t\subset {\rm Fix}(\sigma_t)$ is the fixed component
  of genus at least $2$ for $t\neq 0$, and
  \item $R_0=\lim_{t\to 0}R_t$.
  \end{enumerate}
  
  By \cite[Prop.~6.12]{alexeev2021compact}, it suffices to show
  that if we make a one-parameter deformation the smoothing of $X_0$ into $F_\rho$
  that keeps $X_0$ constant, the limiting curve $R_0$ does not deform,
  up to ${\rm Aut}^0(X_0)$.
  
  The automorphism $\sigma$ on the generic fiber of any smoothing defines a
  birational automorphism of $X$. Any two Kulikov
  models are related by an automorphism followed by
  a sequence of Atiyah flops of types 0, I, II
  along curves in $X_0$ which are either $(-2)$-curves or $(-1)$-curves
  on component(s) of $X_0$. As such, there are only countably
  many curves in $X_0$ along which it is possible to make an Atiyah flop,
  and this continues to be the case after a flop is made.
  Thus, up to conjugation by
  ${\rm Aut}^0(X_0)$, there are only countably many possibilities
  for the birational automorphism $\sigma_0:=\sigma {|}_{X_0}\colon X_0\dashrightarrow X_0$.

   Hence if $X_0\hookrightarrow X$ and $X_0\hookrightarrow \widetilde{X}$
   are smoothings into $F_\rho$ as above,
   we have $\widetilde{\sigma}_0 = \psi\circ \sigma_0\circ \psi\inv$
  for some $\psi\in {\rm Aut}^0(X_0)$.
  
  Let $\{A_j\}$ be the countable set of curves in $X_0$
  along which $\sigma_0$ can be indeterminate.
  Any such curve $A_j$ is ${\rm Aut}^0(X_0)$-invariant. Let $A=\cup_j A_j$
  be their union. Clearly, the limit divisor $R_0$ is contained in the union of
  $A\cup S$ where $S$ is the closure of the fixed locus of $\sigma_0$
  in its locus of determinacy. Similarly, $\widetilde{R}_0$ is contained
  in $A\cup \widetilde{S}$ and $\sigma_0(P)=P$
  if and only if 
  $\widetilde{\sigma}_0(\psi(P)) = \psi(P)$. Since the smoothing $\widetilde{X}$
  is a deformation of the smoothing $X$ and the limiting divisor of $R$
  varies continuously, we conclude that
  $\widetilde{R}_0=\psi(R_0)$ and therefore $R$ is recognizable.
  \end{proof}

\begin{proposition} Any element $(\oX,\epsilon \oR)\in \oF_\rho\uslc$ 
has an automorphism $\overline{\sigma}\in {\rm Aut}(\oX)$. Furthermore, 
$\oR={\rm Fix}(\overline{\sigma})$ and $\overline{\sigma}^*$
acts on $H^0(\oX,\omega_{\oX})\cong \bC$ by multiplication by $\zeta_n$.
  \end{proposition}

\begin{proof} As noted in Remark \ref{sep-moduli}, any point in
 $F_\rho^{\rm sep}=(\bD_\rho \setminus \Delta_\rho)/\Gamma_\rho$
 corresponds to a pair $(\oX,\overline{\sigma})$ of an ADE K3 surface with 
 automorphism, for which $\oR = \Fix(\overline{\sigma})$ is ample and the minimal
 resolution is $\rho$-markable. Then any boundary point
 $(\oX_0,\epsilon \oR_0)\in \oF_\rho\uslc$ is a stable limit of such ADE K3 surface
 pairs $f\colon (\oX,\epsilon \oR) \to C$.
 
 Since $\oR_t$ is $\overline{\sigma}_t$-invariant and the canonical model is unique,
 $\oX$ admits an automorphism $\overline{\sigma}$
 whose fixed locus contains $\oR_0$. In fact, ${\rm Fix}(\overline{\sigma}_0)=\oR_0$:
  ${\rm Fix}(\overline{\sigma})$ is a Cartier divisor, and thus forms a flat family
 of divisors containing $\oR$.
 But ${\rm Fix}(\overline{\sigma}_0)$ already contains the flat limit $\oR_0$. 
The statement about $\omega_{\overline{X}_0}$ follows from the fact that
  $f_*\omega_{\oX/C}$ is invertible (by Base Change and Cohomology,
  since $R^1f_*\omega_{\oX/C}=0$) and $\overline{\sigma}_t^*$ acts
 by $\zeta_n$ on the generic fiber of this line bundle.\end{proof}

\section{Moduli of quotient surfaces}
\label{sec:quotient-surfaces}

We refer the reader to \cite{kollar2013singularities-of-the-minimal}
for the notions appearing in the following definitions.
The pair $(Y,\Delta)$ is called demi-normal if $X$ satisfies Serre's
$S_2$ condition, has double normal crossing singularities in
codimension~$1$,
and $\Delta=\sum d_iD_i$ is an effective Weil
$\bQ$-divisor with $0<d_i\le1$ not containing any components of the
double crossing locus of $Y$.

The following is
\cite[Prop.~2.50(4)]{kollar2013singularities-of-the-minimal}, using
our adopted notations.

\begin{proposition}\label{prop:index-1}
  \'Etale locally, there is a one-to-one correspondence between
  \begin{enumerate}
  \item[(a)] Local demi-normal pairs $(y\in Y, \frac{n-1}n B)$ of
    index~$n$, i.e. such that the divisor $nK_Y+(n-1)B$ is Cartier.
  \item[(b)] Local demi-normal pairs $(\wt y\in\wt Y)$ such that
    $K_{\wt Y}$ is Cartier, with a $\mu_n$-action that is free on a
    dense open subset, and such that the induced action on
    $\omega_{\wt Y}\otimes \bC(\wt y)$ is faithful.
  \end{enumerate}
  Moreover, the pair $(Y,\frac{n-1}n B)$ is slc iff so is $\wt Y$.
\end{proposition}

The variety $\wt Y$ is called the local index-$1$ cover of the pair
$(Y,\frac{n-1}nB)$. \cite[Sec.~2]{kollar2013singularities-of-the-minimal}
also gives a global
construction.

\begin{theorem}
  Let $(\oX,\epsilon\oR)\in\oF_\rho\uslc$ and let $\pi\colon \oX\to
  Y=\oX/\mu_n$ be the quotient map with the branch divisor
  $B=f(\oR)$. Then
  \begin{enumerate}
  \item $nK_Y+(n-1)B\sim 0$,
  \item $B$ and $-K_Y$ are ample $\bQ$-Cartier divisors,
  \item the pair $(Y,\frac{n-1+\epsilon}n B)$ is stable for any
    rational $0<\epsilon\ll1$, i.e. it has slc singularities and the
    $\bQ$-divisor $K_Y + \frac{n-1+\epsilon}n B$ is ample.
  \end{enumerate}
  Vice versa, for a pair $(Y,B)$ satisfying the above conditions, its
  index-$1$ cover $\oX$ with the ramification divisor $\oR$ satisfies:
  \begin{enumerate}
  \item $K_\oX\sim 0$ and the $\mu_n$-action on $\oX$ is non-symplectic,
  \item $\oR$ is $\bQ$-Cartier,
  \item the pair $(\oX,\epsilon\oR)$ is stable for any rational
    $0<\epsilon\ll1$.
  \end{enumerate}
\end{theorem}
\begin{proof}
  Follows from the above Proposition~\ref{prop:index-1} and the
  formulas
  \begin{displaymath}
    \pi^*(B)=n\oR,\qquad
    \pi^*\left( K_Y + \frac{n-1+\epsilon}n B \right) = K_\oX + \epsilon\oR.    
  \end{displaymath}
\end{proof}

\begin{corollary}
  The coarse moduli space $\oF_\rho\uslc$ coincides with the
  normalization of the KSBA compactification of the irreducible
  component in the moduli space of
  the log canonical pairs $(Y,\frac{n-1+\epsilon}n B)$ of log del
  Pezzo surfaces $Y$ with $(n-1)B\in |-nK_Y|$ in which a generic
  surface is a quotient of a K3 surface with a non-symplectic
  automorphism of type $\rho$.
  The stack for the former is a $\mu_n$-gerbe over the stack for the latter.
\end{corollary}

For the proof, we note that a small deformation of a K3 surface is a K3 surface.

\begin{example}
  The KSBA compactification moduli of K3 surfaces of degree~$2$ for
  the ramification divisor $R$ constructed in
  \cite{alexeev2019stable-pair} is equivalent to the Hacking's
  compactification \cite{hacking2004compact-moduli} of the moduli
  space of pairs $(\bP^2, \frac{1+\epsilon}2 B_6)$ of plane sextic
  curves.
\end{example}

\section{Extensions}
\label{sec:extensions}

The results of this paper are easily extended to the case of a
nonsymplectic action by an arbitrary finite group $G$ and to more
general divisors defined by group actions. Most of the changes amount to
introducing more cumbersome notations. 

\subsection{A general nonsymplectic group of automorphisms}

\begin{definition}
  Let $X$ be a smooth K3 surface and $\sigma\colon G\subset\Aut X$ be a finite
  symmetry group. The action of $G$ on $H^{2,0}(X)=\bC \omega_X$ gives
  the exact sequence
  \begin{displaymath}
    0 \to G_0 \to G \xrightarrow{\alpha} \mu_n\to 1, \qquad \mu_n \subset \bC^*.
  \end{displaymath}
  One says that $G$ is nonsymplectic (or not purely symplectic) if
  $G\ne G_0$, i.e. $\alpha\ne 1$.
\end{definition}

We now extend the results of Section~\ref{sec:moduli} directly to this more
general setting.

\begin{definition}
  Fix a finite subgroup $\rho\colon G\to O(\lk)$ and a nontrivial
  character $\chi\colon G\to\bC^*$. Let
  $(X,\sigma\colon G\to\Aut X)$ be a K3 surface with a non-symplectic
  automorphism group.

  A {\it $(\rho,\chi)$-marking} of $(X,\sigma)$ is an isometry
  $\phi:H^2(X,\bZ)\to \lk$ such that for any $g\in G$ one has
  $\phi\circ\sigma(g)^* = \rho(g)\circ\phi$ and such that
  the character $\alpha\colon G\to \bC^*$ induced by $\sigma$
  coincides with $\chi$. 
  We say that $(X,\sigma)$ is
  {\it $\rho$-markable} if it admits a $\rho$-marking.

  A family of $(\rho,\chi)$-marked K3 surfaces is a smooth family
  $f\colon (\cX,\sigma_B,\phi_B)\to B$ with a group of automorphisms
  $\sigma_B\colon G\to \Aut(\cX/B)$ and with a marking
  $\phi_B\colon R^2f_*\bZ \to L\otimes\underline{\bZ}_B$ such that
  every fiber is a $(\rho,\chi)$-marked K3 surface.

  A family of smooth $\rho$-markable K3 surfaces is a family
  $f\colon (\cX,\sigma_B)\to B$ of K3 surfaces with a group of
  automorphisms over base $B$ which admits a $\rho$-marking locally on
  $B$.

  We define the moduli stacks $\cM_{\rho,\chi}$ of
  $(\rho,\chi)$-marked, resp. $F_{\rho,\chi}$ of
  $(\rho,\chi)$-markable K3 by taking $\cM_{\rho,\chi}(B)$,
  resp. $F_{\rho,\chi}(B)$ to be the groupoids of such
  families over~$B$.
\end{definition}

\begin{definition}
  Define the vector space 
  \begin{displaymath} 
    L_\bC^{\rho,\chi} = \{x\in L_\bC \mid \rho(g)(x) = \chi(g)x \} 
  \end{displaymath}
  to be the intersection
  of the eigenspaces for the individual $g\in G$, 
  and the period domain as 
  \begin{displaymath}
    \bD_{\rho,\chi} = \bP\{ x\in L_\bC^{\rho,\chi} \mid x\cdot \bar x > 0\}
  \end{displaymath}
  The second condition is redundant if there exists $g\in G$ with
  $\chi(g)>2$. Thus, $\bD_\rho$ is a type IV domain if $|\chi(G)|=2$
  and a complex ball, a type I domain if $|\chi(G)|>2$.

  The discriminant locus is 
  $\Delta_\rho:= \cup_{\delta}\delta^\perp\cap\Delta_\rho$ ranging
  over all roots $\delta$ in $(L^G)^\perp$, where $L^G= \{a\in L \mid
  \rho(g)(a) = a \}$ is the sublattice of $L$ fixed by $G$.
\end{definition}

\begin{definition}
  The group of changes-of-marking is
  $$\Gamma_\rho := \{\gamma\in O(L) \mid \gamma\circ\rho = \rho\circ
  \gamma\}.$$
\end{definition}

Then the direct analogue of Lemma~\ref{lem:im-pi-rho} and
Theorem~\ref{separated-rho} is

\begin{theorem}
  For a fixed finite group $\rho\colon G\to O(L)$ with a nontrivial
  character $\chi\colon G\to \bC^*$:
  \begin{enumerate}
  \item There exists a fine moduli space $\cM_{\rho,\chi}$ of $(\rho,\chi)$-marked 
    K3 surfaces $(X,\sigma,\phi)$.
    It admits an \'etale period map
    $\pi_\rho\colon \cM_{\rho,\chi}\to \bD_{\rho,\chi}\setminus\Delta_\rho$. The
    fiber $\pi_\rho\inv(x)$ over a point $x\in\bD_{\rho,\chi}\setminus
    \Delta_\rho$ is a torsor over $\Gamma_\rho\cap (\bZ_2\cap W_x)$.
  \item The moduli stack of $\rho$-markable K3 surfaces $(X,\sigma)$
    is obtained as a quotient of $F_{\rho,\chi}$ by $\Gamma_\rho$. On the
    level of coarse moduli spaces it admits a bijective map to
    $(\bD_{\rho,\chi}\setminus \Delta_\rho)/\Gamma_\rho$. 
  \end{enumerate}
\end{theorem}
\begin{proof}
  If the group $G$ does not act purely symplectically, i.e. there
  exists $g\in G$ with $\rho(g)(x)\ne x$ then $L^G\perp x$ and
  $S_X^G\simeq L^G$. The rest of the proof of
  Lemma~\ref{lem:im-pi-rho} works the same for any finite group. 
  And the proof of Theorem~\ref{separated-rho} goes through verbatim.
\end{proof}

\subsection{More general polarizing divisors}

With a more general group action, there are more choices for the
polarizing divisors. For a generic K3 surface $X$ with a period
$x\in\bD_{\rho,\chi}\setminus \Delta_\rho$ we can consider any
combination $\sum b_i B_i$ of curves $B_i$ which are either fixed by some
element $g\in G$ or are some of the $(-2)$-curves corresponding to the
roots in the generic Picard lattice $(L_\bC^{\rho,\chi})^\perp \cap
L$ that generically gives a big and nef divisor on
$X$. Theorem~\ref{thm:main} extends immediately to this situation with
the same proof.

\bibliographystyle{amsalpha}

\def\cprime{$'$}
\providecommand{\bysame}{\leavevmode\hbox to3em{\hrulefill}\thinspace}
\providecommand{\MR}{\relax\ifhmode\unskip\space\fi MR }
\providecommand{\MRhref}[2]{%
  \href{http://www.ams.org/mathscinet-getitem?mr=#1}{#2}
}
\providecommand{\href}[2]{#2}

\end{document}